\documentclass[12pt]{article}
\usepackage[english]{babel}
\usepackage[cp1251]{inputenc}
\usepackage{amsmath,amssymb,graphicx,amsthm,wrapfig,enumerate,comment}
\usepackage[unicode, pdftex]{hyperref}
\newtheorem{theorem}{Theorem}

\newtheorem{lemma}{Lemma}
\newtheorem{proposition}{Proposition}
\newtheorem*{question}{Question}

\theoremstyle{remark}
\newtheorem{remark}{Remark}

\theoremstyle{definition}
\newtheorem{definition}{Definition}

\newcommand{\R}{\mathbb{R}}

\newcommand{\Z}{\mathbb{Z}}
\newcommand{\ZZ}{\widetilde{\mathbb{Z}}}
\newcommand{\del}{\smash{\mskip3mu\lower1truept\hbox{\vdots}\mskip3mu}}
\newcommand{\es}{\varnothing}
\renewcommand{\phi}{\varphi}
\newcommand{\Id}{\mathrm{Id}}

\newcommand{\sk}{\mathrm{sk}}
\newcommand{\Iso}{\mathrm{Iso}}

\newcommand{\ee}{\widetilde e}

\begin{document}

\author{Andrey Ryabichev}
\title{Eliashberg's $h$-principle and generic maps of surfaces with prescribed singular locus%
\footnote{Keywords: stable map; fold; cusp; $h$-principle.}}
\date{}


\maketitle

\begin{abstract}
We extend Y.\,Eliashberg's $h$-principle to smooth maps of surfaces which are allowed to have cusp singularities, as well as folds. More precisely, we prove a necessary and sufficient condition for a given map of surfaces to be homotopic to one with given loci of folds and cusps. Then we use these results to obtain a necessary and sufficient condition for a subset of a surface $M$ to be realizable as the critical set of some generic smooth map from $M$ to a given surface~$N$.
\end{abstract}

\section{Introduction}

All manifolds and maps between them are assumed to be infinitely smooth unless we explicitly specify otherwise.
Also all manifolds are assumed to have no boundary and to be equipped with a Riemannian structure.
All submanifolds we consider will be closed without boundary but not necessary compact.
The word ``homologous'' always means ``homologous modulo two''.
If $Y$ is a codimension $k$ submanifold of $X$, then we denote
the Poincar\'e dual class of $Y$ in $H^k(X,\Z_2)$ by $[Y]$.
The symbol ``$\simeq$'' means ``isomorphic'' (groups, vector bundles, etc.).
The symbol ``$=$'' for topological spaces means ``homeomorphic''.


\vspace{.5em}

Let $M$ and $N$ be connected compact 2-manifolds.
We will call such manifolds {\it closed surfaces}.
A map $f:M\to N$ is called {\it generic} if all critical points of $f$ are folds or cusps.
Such maps form a dense open subset in the space of smooth maps from $M$ to $N$, see~\cite{whitney} 
 for details.

\begin{figure}[h]
\center{\includegraphics{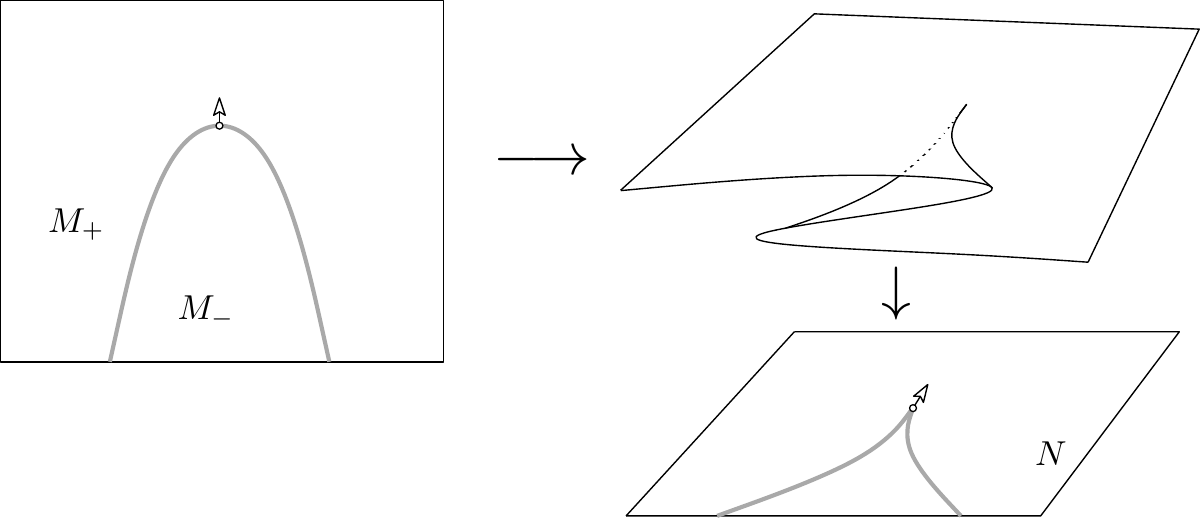}}
\caption{The direction vector of a negative cusp.}\label{fig:cusp-vector}
\end{figure}

For a map $f:M\to N$ we denote its set of critical points by $\Sigma(f)$.
If $f$ is generic, then $\Sigma(f)\subset M$ is a smooth closed 1-submanifold and
the set of cusp points is a discrete subset of $\Sigma(f)$.
For every cusp point we can choose a unit normal vector to $\Sigma(f)$ such that
the image of this vector under $df$ is directed outwards with respect to the cusp, see~Fig.~\ref{fig:cusp-vector}.
Such a vector is called {\it the direction of the cusp point}. 

Suppose we are given a smooth 1-submanifold $C\subset M$ and a discrete subset $P\subset C$.
Suppose at every point of $P$ a unit normal vector to $C$ is chosen.
We call such vectors {\it directions}.
Then a {\it $(C,P)$-immersion} is a generic map $f:M\to N$
with the set of folds equal $C\setminus P$
and the set of cusps (with given directions) equal $P$.
In this paper we answer the following question.

\begin{question}
Under which conditions on a nonempty closed 1-submanifold $C\subset M$ and
on a discrete subset $P\subset C$ with given directions is there a $(C,P)$-immersion $f:M\to N$?
\end{question}

If $C=\es$, the answer of this question is well known since in that case $f$ must be just a covering map. So from now on we assume $C\ne\es$.

Suppose $[C]=0\in H^1(M;\Z_2)$.
Then clearly $C$ separates $M$ into two parts, say $M_+$ and $M_-$, such that $\partial M_+=\partial M_-=C$.
This allows us to distinguish between {\it positive} and {\it negative} elements of $P$ when $[C]=0$.
More precisely, an element $p\in P$ is called {\it positive} if its direction vector is directed into $M_-$,
or {\it negative} if its direction vector is directed into $M_+$.
We set $n_+$ (respectively $n_-$) to be the number of the positive (respectively negative) elements of $P$.
Note that to define $n_+$ and $n_-$ we need to choose which of the parts of $M\setminus C$ is $M_+$ and which one is $M_-$.
Similarly, we can distinguish between {\it positive cusps} and {\it negative cusps} of a generic map~$f$ if $[\Sigma(f)]=0$.
For example, the cusp in Figure~\ref{fig:cusp-vector} is negative.


\begin{theorem}\label{th:homotopic}
Suppose we are given closed surfaces $M,N$, a nonempty closed 1-submanifold $C\subset M$,
a discrete subset $P\subset C$ with chosen directions and a continuous map $f:M\to N$.
Then $f$ is homotopic to a $(C,P)$-immersion if and only if all of the following conditions hold:
\begin{enumerate}
\item[1.1] $[C]=w_1(M)+f^*w_1(N)\in H^1(M;\Z_2)$;
\item[1.2] $[P]=w_2(M)+f^*w_2(N)\in H^2(M,\Z_2)$;
\item[1.3] if $[C]=0$, then $\big|\chi(M_+)-\chi(M_-)-n_+ +n_-\big| = |\deg f\cdot\chi(N)|$.
\end{enumerate}
\end{theorem}

Here $\deg f$ is the degree of a map of possibly non\-orientable manifolds.
We can define it when $f^*w_1(N)=w_1(M)$, see Definition~\ref{def:degree} in \S\ref{s:twisted-euler-class}.
So if $[C]=0$, then $\deg f$ is well defined as soon as condition 1.1 holds.

Note that if we swap $M_+$ and $M_-$, then the values of $n_+$ and $n_-$ will be interchanged.
Therefore $\big|\chi(M_+)-\chi(M_-)-n_+ +n_-\big|$ does not depend on
which of the two submanifolds that are bounded by $C$ is $M_+$ and which is $M_-$.

\begin{remark}
Y.~Eliashberg proved Theorem~\ref{th:homotopic} in the case when $N$ is orientable, see \cite[Th.~4.8,~4.9]{eliashberg}.
In our proof we use Eliashberg's $h$-principle for maps with fold type singularities, see~\S\ref{s:eliashbergs-theorem}.
But in the case when $N$ is nonorientable the computation of the obstruction  turns out to be quite a bit more tricky.
In particular, in order to handle this case we introduce the notion of twisted tangent bundle, see~\S\ref{s:twisted-tangent-bundle}.
\end{remark}

As a consequence of Theorem \ref{th:homotopic} we obtain the following answer to the question above.

\begin{theorem}\label{th:existence}
Suppose we are given closed surfaces $M,N$, a nonempty closed 1-submanifold $C\subset M$
and a discrete subset $P\subset C$ with chosen directions.
Then there is a $(C,P)$-im\-mer\-sion $f:M\to N$ if and only if all of the following conditions hold:
\begin{enumerate}
\item[2.1] $[P]=[C]^2$;

\item[2.2] if $w_1(N)=0$, then $[C]=w_1(M)$;

\item[2.3] if $[C]=0$, then there exists $d\in\Z$ such that
    $\big|\chi(M_+)-\chi(M_-)-n_+ +n_-\big| = |d\cdot\chi(N)|$\ \ and\ \ $\chi(M)\le |d|\cdot\chi(N)$;\ \
     $d$ must be even if $M$ is orientable and $N$ is nonorientable;

\item[2.4] if $[C]\ne0$, $w_1(M)\ne0$ and $[C]^2\ne w_2(M)$,
    then $w_2(N)\ne0$ and $\chi(N)>\chi(M)$.
\end{enumerate}
\end{theorem}

This result can be viewed as a first step in the study of the relationship between a generic map, its set of critical points
and the set of critical values (called {\it the apparent contour}).
Our theorem generalizes the results of \cite{hmjrf}, \cite{mj}, \cite{mjrf} and \cite{m.yamamoto} on the combinatorial description of the critical locus of a generic map.
However, 
finding restrictions on possible
apparent contours
is much more complicated, see~\cite[\S6]{yamamoto}.
The next step would be to find an obstruction for a continuous map to be homotopic to a generic one with a given set of critical values.

\subsubsection*{Structure of the paper}

In \S\ref{s:eliashbergs-theorem} we recall the results from \cite{eliashberg} and \cite{hirsch}
which will be used in the proof of Theorem~\ref{th:homotopic}.
In \S\ref{s:local-systems} we discuss  local systems and
prove some technical results on cohomology of manifolds with local coefficients.
In \S\ref{s:twisted-euler-class} we define the Euler class with local coefficients.
We use it to classify rank $2$ vector bundles up to isomorphism in \S\ref{s:2-bundles-classification}.

In \S\ref{s:twisted-tangent-bundle} we describe the vector bundle $f^*(TN)$
for a $(C,P)$-immersion $f:M\to N$ and compute its characteristic classes in terms of $C$ and $P$.
Using this description we prove Theorem~\ref{th:homotopic} in \S\ref{s:proof-homotopic}.
In \S\ref{s:curves} we prove a few preliminary statements about simple closed curves and maps of surfaces.
Finally, in \S\ref{s:proof-existence} we deduce Theorem~\ref{th:existence} from Theorem~\ref{th:homotopic}.

\subsubsection*{Acknowledgments}
I wish to thank Alexey Gorinov for helpful suggestions and conversations.

\section{Preliminaries on manifolds and vector bundles}

In this section we state some general theorems about $n$-manifolds.
Note that we use these results only for $n=2$.

\subsection{Eliashberg's and Hirsch's Theorems}\label{s:eliashbergs-theorem}

If $p:E\to X$ and $p':E'\to Y$ are vector bundles over topological spaces,
$f:X\to Y$ is a continuous map and $\phi:E\to E'$ is a fiberwise map such that $p'\circ\phi=f\circ p$,
then we say that {\it $\phi$ covers $f$}.

In our proof of Theorem~\ref{th:homotopic} we will use 
\cite[
Th.\,2.2]{eliashberg}.
Let us briefly recall the notation.
We set $M,N$ to be $n$-manifolds, $n\ge2$, and
let $C\subset M$ be a codimension 1 closed submanifold. 

A {\it $C$-immersion $M\to N$} is a map with fold type singularities along $C$ and no other critical points.
A {\it $C$-monomorphism $M\to N$} is a fiberwise morphism of tangent bundles $\phi:TM\to TN$ such that
the restrictions $\phi|_{T(M\setminus C)}$ and $\phi|_{TC}$ are fiberwise injective
and there exist a tubular neighborhood $U\supset C$ and a nonidentity involution $h:U\to U$ preserving $C$
such that over $U$ we have $\phi\circ dh=\phi$.
(Instead of $C$-monomorphisms one can consider fiberwise isomorphisms $T^CM\to TN$,
see \S\ref{s:tcm} for the definition.)

Let $V\subset M$ be a closed submanifold of dimension $n$ with boundary
such that every component of $M\setminus V$ intersects $C$ and let $f_0:V\to N$ be a smooth map.
Denote by $\mathrm{Imm}_{C,V}(M,N)$ the space of maps $f:M\to N$ such that
$f|_V=f_0$ and $f|_{M\setminus V}$ is a $C$-immersion.
Let $\mathrm{Mon}_{C,V}(M,N)$ be the space of fiberwise maps $\phi:TM\to TN$ such that
$\phi|_{TV}=df_0$ and $\phi|_{T(M\setminus V)}$ is a $C$-monomorphism.

\begin{theorem}[Eliashberg]\label{th:eliash}
The map
$\pi_0\big(\mathrm{Imm}_{C,V}(M,N)\big) \to \pi_0\big(\mathrm{Mon}_{C,V}(M,N)\big)$
induced by taking the differential is surjective.
\hfill$\square$
%
\end{theorem}

In other words, every $\phi\in\mathrm{Mon}_{C,V}(M,N)$
can be deformed to the differential of a $C$-immersion
in the class of $C$-monomorphisms fixed over $V$.
In particular, if $\mathrm{Mon}_{C,V}(M,N)$ is nonempty, then $\mathrm{Imm}_{C,V}(M,N)$ is nonempty.
So instead of constructing a $(C,P)$-immersion
it would suffice to construct just a morphism of tangent bundles that would be a fiberwise isomorphism outside $C$,
and such that its restriction to some neighborhood of $C$ would be the differential of a $(C,P)$-immersion.

\vspace{.5em}

We will also use 
\cite[p.\,265, Th.\,5.7]{hirsch}.
This statement is known as the {\it relative $h$-principle for immersions in positive codimension}.

\begin{theorem}[Hirsch]\label{th:hirsch}
Let $X,Y$ be manifolds such that $\dim X<\dim Y$.
Suppose we are given a smooth map $f:X\to Y$ and a fiberwise injective homomorphism $\phi:TX\to TY$ that covers $f$ and
such that for some $CW$-subcomplex $Z\subset X$ we have $\phi|_Z=df|_Z$.
Then there exists an immersion $f':X\to Y$ such that
$\phi$ is homotopic to $df'$ in the class of fiberwise injective homomorphisms fixed over $Z$.
\hfill$\square$
\end{theorem}

\subsection{Orientation local systems of vector bundles}\label{s:local-systems}


We denote the orientation local system of a vector bundle $E$ 
by $\ZZ_E$,
and let $\ZZ_M$ be the orientation sheaf of a manifold $M$.
Note that we mostly use local systems as coefficients for {\it singular} cohomology.
So we think of a local system as a module over the path groupoid of $M$ (see e.\,g.\ \cite[pp.179--180]{spanier}).

In \S\S\ref{s:local-systems} and \ref{s:twisted-euler-class} we set $M$ and $N$ to be closed connected $n$-manifolds.

\begin{proposition}\label{pr:n-th-cohomology}
Let $E\to M$ be a 
vector bundle.
If $w_1(E)=w_1(M)$, then we have $H^n(M;\ZZ_E)\simeq\Z$. 
In particular, $H^n(M;\ZZ_M)\simeq\Z$. \
If $w_1(E)\ne w_1(M)$, then $H^n(M;\ZZ_E)\simeq\Z_2$.
\end{proposition}

\begin{proof}
The local system $\ZZ_E$ is uniquely determined by the class $w_1(E)$.
Namely, a loop $\gamma$ acts on the fiber of $\ZZ_E$ as 
$(-1)^{ w_1(E)\frown\gamma}$.
Note that in particular $\ZZ_E\simeq\ZZ_M$ if and only if $w_1(E)=w_1(M)$.

Next we use the Poincar\'e duality:
for every local system $L$ on $M$ we have
$$H_k(M;\ZZ_M\otimes L)\simeq H^{n-k}(M;L)$$
(see e.\,g.\ \cite[Th.\,10.2]{spanier}).
We set $L=\ZZ_E$ and $k=0$.

If $w_1(E)=w_1(M)$, then the local systems $\ZZ_M$ and $\ZZ_E$ are isomorphic.
So their tensor product is constant and we have $H^n(M;\ZZ_E)\simeq H_0(M;\Z)\simeq\Z$.

If $w_1(E)\ne w_1(M)$, then there is a loop which either changes the orientation of $M$ and preserves the orientation of~$E$, or vice versa.
If we consider such a loop as an element of the singular chain group $C_1(M;\ZZ_M\otimes\ZZ_E)$,
then its boundary is equal to twice its basepoint.
On the other hand, the boundary of every $1$-chain reduced modulo $2$ is zero.
Therefore $H^n(M;\ZZ_E)\simeq H_0(M;\ZZ_M\otimes\ZZ_E)\simeq\Z_2$.
\end{proof}

\begin{proposition}\label{pr:preserving-orientation}
Suppose we are given a map $f:M\to N$ such that $w_1(M)=f^*w_1(N)$.
Then $f$ induces a homomorphism $f^*:H^n(N;\ZZ_N)\to H^n(M;\ZZ_M)$ well defined up to a sign.
\end{proposition}

\begin{proof}
Note that for all topological spaces $X,Y$ and every local system $L$ over $Y$,
any map $g:X\to Y$ induces a homomorphism $g^*:H^*(Y;L)\to H^*(X;g^*L)$.
Also note that for any vector bundle $E\to Y$ we have $g^*(\ZZ_E)=\ZZ_{g^*(E)}$.


Since $w_1(M)=f^*w_1(N)$, the local system $\ZZ_M$ is isomorphic to $\ZZ_{f^*(TN)}=f^*(\ZZ_N)$.
Therefore we can define a homomorphism $H^n(N;\ZZ_N)\to H^n(M;\ZZ_M)$ induced by $f$.
This homomorphism is well defined up to a sign, which depends on the choice of the isomorphism $\ZZ_M\simeq f^*(\ZZ_N)$.
\end{proof}

\begin{definition}\label{def:degree}
Let $f:M\to N$ be a map such that $w_1(M)=f^*w_1(N)$.
Then by Proposition~\ref{pr:preserving-orientation}
there is a homomorphism $f^*:H^n(N,\ZZ_N)\to H^n(M,\ZZ_M)$.
By Proposition~\ref{pr:n-th-cohomology} we have $H^n(M,\ZZ_M)\simeq H^n(N,\ZZ_N)\simeq\Z$.
So $f^*:\Z\to\Z$ is multiplication by some integer~$d$
(which is well defined up to a sign).
Then {\it the degree of the map $f$} is the absolute value $|d|$.
We denote it by $\deg f$.
\end{definition}

\subsection{Euler class with local coefficients}\label{s:twisted-euler-class}

One can define the Euler class as the first obstruction to the existence of a nonzero section.
Such an approach uses the obstruction theory with local coefficients over $CW$-complexes.
Let us sketch this definition
(see e.\,g.\ \cite[\S12]{milnor-stasheff} or \cite[chapter\,VI,~\S5]{whitehead} for details).

Let $E\to X$ be a rank $m$ vector bundle over a $CW$-complex $X$.
There is a section of 
the corresponding sphere
bundle $S(E)$ over $\sk^{m-2}(X)$.
Such a section is unique up to homotopy.
Every section of $S(E)$ over $\sk^{m-2}(X)$ can be extended to $\sk^{m-1}(X)$.
Then 
a trivialization of $E$ over an $m$-cell $B\subset X$
gives us a map $o_B:\partial B\to S^{m-1}$, which may be non-contractible.
If we choose orientations of $B$ and of the fiber of $S(E)$ over an interior point of $B$,
then the integer $\deg(o_B)$ is well defined.
The correspondence $B\mapsto\deg(o_B)$ is a cocycle with coefficients in $\ZZ_E$.
If we change the section on $(m-1)$-cells,
the cohomology class of this cocycle remains the same.

However to compute the Euler class for a given vector bundle over a manifold of the same dimension
the following equivalent approach turns out to be more convenient.

\begin{definition}\label{def:euler-class}
Suppose we are given a rank $n$ vector bundle $E$ over an $n$-manifold $M$.
Take a section $s$ with a discrete set of zeroes.
For each zero $x_i$ take its neighborhood $D_i\subset M$ homeomorphic to a disk.
Choose an orientation of $D_i$ and an orientation of the fiber $E|_{x_i}$.
If we trivialize the restriction $E|_{D_i}$,
then the section $s$ over $\partial D_i$
determines a map
$$\partial D_i=S^{n-1}\to S^{n-1}=S(E|_{x_i}),$$
where $S(E|_{x_i})$ is the unit sphere in the fiber.
Denote the degree of this map by $a_i$.

The number $a_i$ is called {\it the index of the zero point $x_i$}.
Note that we do not assume that $x_i$ is non-degenerate, so $a_i$ can be an arbitrary integer,
and the sign of $a_i$ depends on the choices of the orientations of $D_i$ and $E|_{x_i}$.
The sum $\sum a_ix_i$  
is a singular cycle in
$H_0(M;\ZZ_M\otimes\ZZ_E)$,
and its Poincar\'e dual is called {\it the Euler class} $\ee(E)\in H^n(M;\ZZ_E)$.
\end{definition}

If we assume that $x_i\notin\sk^{n-1}(M)$,
then our definition and the obstruction-theoretical one give the same cocycle.
Therefore $\ee(E)$ does not depend on the choice of the section $s$.


Note that the Euler class is functorial: $\ee(f^*(E))=f^*(\ee(E))$ for any map $f:M'\to M$.
We will speak of $\ee(E)$ as an integer
defined up to a sign, respectively an integer modulo $2$,
if $w_1(E)=w_1(M)$, respectively if $w_1(E)\ne w_1(M)$ (see Proposition~\ref{pr:n-th-cohomology}).

\begin{proposition}\label{pr:ee-tm-equals-chi-m}
For every closed $n$-manifold $M$ we have $\ee(TM)=\pm\chi(M)$.
\end{proposition}

\begin{proof}
A local orientation of the manifold and a local orientation of the fibers of its tangent bundle naturally determine one another.
In other words, for $E=TM$ we have a {\it canonical} isomorphism $\ZZ_M\otimes\ZZ_E\simeq\Z$.
Therefore the sum of the indices of the zeroes of any section of $TM$ is a well defined integer. 
We call this sum {\it the total index of the section}.

So we can 
use the same strategy as in orientable case (see e.\,g.\ \cite[chapter~5, \S4.2]{prasolov}):
triangulate $M$ and consider a vector field which has a single zero of index $(-1)^k$ at the center of every $k$-simplex.
Then the 
total index of this vector field equals $\chi(M)$.

Note that the total index is an invariant of the vector bundle $TM$,
so for any other vector field with isolated zeroes the total index still equals $\chi(M)$.
\end{proof}

\begin{proposition}\label{pr:ee-equals-w-mod2}
The reduction $H^n(M;\ZZ_E)\to H^n(M;\Z_2)$ maps $\ee(E)$ to $w_n(E)$.
\end{proposition}

\begin{proof}
The $n$-th Stiefel-Whitney class of $E$ can be viewed as an obstruction
to the existence of a section of the sphere bundle $S(E)$ reduced modulo $2$, see e.\,g.~\cite[Th.\,12.1]{milnor-stasheff}

Note that if $w_1(E)\ne w_1(M)$, then the reduction $H^n(M;\ZZ_E)\to H^n(M;\Z_2)$ is an isomorphism,
so in this case one can identify $\ee(E)$ with $w_n(E)$.
\end{proof}

\subsection{Classification of rank 2 vector bundles over surfaces}\label{s:2-bundles-classification}

\begin{lemma}\label{l:ee-universal}
Suppose $M$ is a closed surface.
Then every two rank $2$ vector bundles $E,E'\to M$
such that $w_1(E)=w_1(E')$ and $\ee(E)=\pm\ee(E')$, are isomorphic.
\end{lemma}

\begin{proof}
Since $w_1(E)=w_1(E')$, we have an isomorphism $\ZZ_{E}\simeq\ZZ_{E'}$.
Choose it so that the induced isomorphism $H^2(M;\ZZ_{E})\to H^2(M;\ZZ_{E'})$ takes $\ee(E)$ to $\ee(E')$.
Fix a point $x_0\in M$ and a small disk $D_0\ni x_0$.
Then take an isomorphism of the restrictions $\phi:E|_{D_0}\to E'|_{D_0}$
which is compatible with the isomorphism of the fibers of $\ZZ_{E}$ and $\ZZ_{E'}$ at $x_0$.

There is a section $s:M\to E$ with a single zero point at $x_0$.
(If $w_1(E)=w_1(M)$, then the zero point $x_0$ has index $\ee(E)$.
If $w_1(E)\ne w_1(M)$ and $H^2(M;\ZZ_E)\simeq\Z_2$,
then 
for every $k\equiv\ee(E)\mod 2$
there is a section of $E$ that has a single zero of index $k$.)

By obstruction theory, the section of $E'$ over $D_0$
obtained as the composition $\phi\circ s|_{D_0}$
can be extended to a section $s'$ on the whole of $M$ with no zeroes but $x_0$.


The bundles $E|_{M\setminus\{x_0\}}$ and $E'|_{M\setminus\{x_0\}}$ have trivial line subbundles spanned by $s$ and $s'$.
Denote them by $L$, respectively $L'$.
They are both trivialized by $s$ and $s'$, so there is an isomorphism $L\to L'$ compatible with $\phi$ on $D_0\setminus\{x_0\}$.
Similarly, since $w_1(E)=w_1(E')$, there is an isomorphism of the quotient bundles
$(E|_{M\setminus\{x_0\}})/L\to (E'|_{M\setminus\{x_0\}})/L'$,
and this isomorphism can be chosen to be compatible with $\phi$ on $D_0\setminus\{x_0\}$.
These two isomorphisms together with $\phi$ define an isomorphism $E\to E'$ over the whole $M$.
\end{proof}

Let us sketch another proof of Lemma~\ref{l:ee-universal} which can be 
generalized to 
higher dimensions.

The vector bundles $E$ and $E'$ are isomorphic if and only if
the corresponding maps $\epsilon,\epsilon':M\to BO(2)$ 
 are homotopic.
 Suppose $M$ has been equipped with a $CW$-structure.
The first obstruction to a homotopy between $\epsilon$ and $\epsilon'$ is an element of $H^1(M;\pi_1(BO(2)))=H^1(M;\Z_2)$.
One can prove that this element is $w_1(E)-w_1(E')$, see e.\,g.~\cite[chapter\,VI,~\S6]{whitehead}.

Suppose we have chosen an isomorphism of $E$ and $E'$ restricted to the 1-skeleton of $M$.
The second obstruction is an element of $H^2(M;\{\pi_2(BO(2))\})$.
Here $\{\pi_2(BO(2))\}$ is a $\pi_1(M)$-module, i.\,e.\ a local system on $M$ with fiber $\Z\simeq \pi_2(BO(2))$
(the actions of $\pi_1(M)$ induced by $\epsilon$ and $\epsilon'$ are the same since $w_1(E)=w_1(E')$).
It turns out that the second obstruction is $\ee(E)-\ee(E')$.

\section{Twisted tangent bundle}\label{s:twisted-tangent-bundle}

Suppose we are given closed surfaces $M,N$ and a $(C,P)$-immersion $f:M\to N$.
We will describe 
the vector bundle $f^*(TN)$ in terms of $C$ and $P$.
To this end we will consider an operation of regluing a vector bundle 
over a codimension~1 submanifold of the base. 
We call such an operation {\it a~modification} of the vector bundle.

\subsection{Folding modification of the tangent bundle}\label{s:tcm}

In this subsection we let $M,N$ be closed $n$-manifolds and let $C\subset M$ be a closed submanifold of codimension~$1$.
We will describe $f^*(TN)$ for a $C$-immersion $f:M\to N$
and then compute its characteristic classes in terms of $[C]$.

\begin{definition}\label{def:tcm}
Let $T^C M$ be the rank $n$ vector bundle obtained by regluing $TM$ along $C$
using the nonidentity orthogonal involution of $TM|_C$ which is identity on $TC$.
We canonically identify $TM$ and $T^C M$ over $M\setminus C$. 
Also, we fix a natural imbedding $TC\subset T^C M$.
\end{definition}

Given $C\subset M$, the isomorphism class of the vector bundle $T^CM$ does not depend on the Riemannian metric on $M$.
Indeed, let $E_0$, respectively $E_1$, be the bundle $T^CM$ constructed using the Riemannian metric $g_0$, respectively $g_1$.
Note that $g_t=(1-t)g_0+tg_1$ is a Riemannian metric for all $t\in[0;1]$.
Using $g_t$ one can construct a rank $n$ vector bundle $E\to M\times[0;1]$
such that $E|_{M\times 0}\simeq E_0$ and $E|_{M\times 1}\simeq E_1$.
Therefore $E_0$ and $E_1$ are isomorphic.

\begin{figure}[h]
\center{\includegraphics
{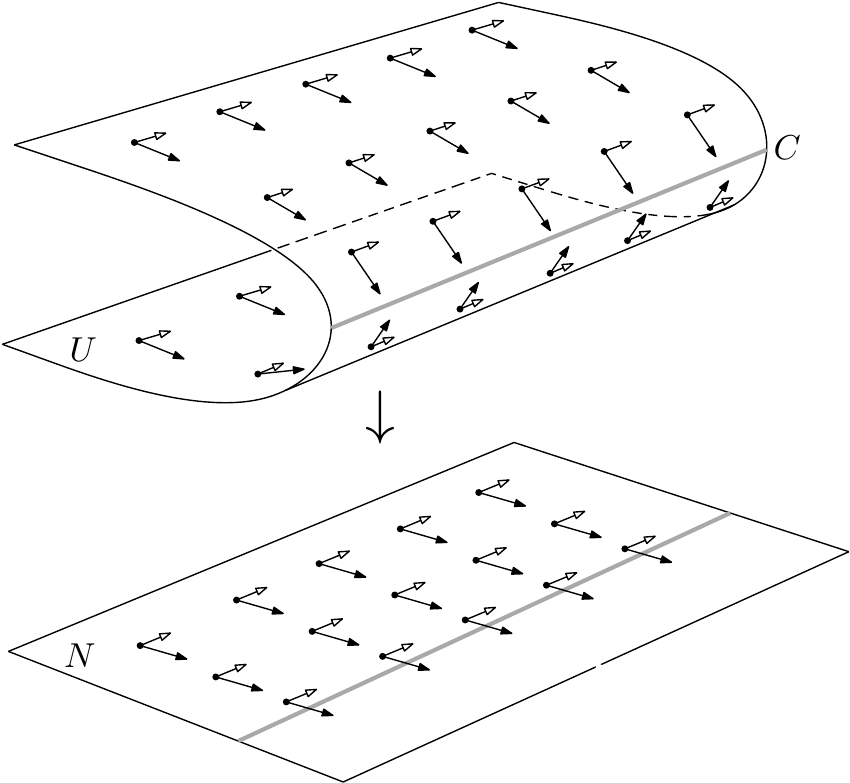}}
\caption{A trivialization of $T^CM|_U$ and its image under $\phi$ in $TN$.}\label{fig:tcm-frames}
\end{figure}

\begin{proposition}\label{pr:tcm-isomorphism}
If $f:M\to N$ is a $C$-immersion, then for every tubular neighborhood $U\supset C$
there exists a fiberwise isomorphism $\phi:T^C M\to TN$ covering $f$ such that $\phi|_{T(M\setminus U)}=df|_{T(M\setminus U)}$ and $\phi|_{TC}=df|_{TC}$
\end{proposition}

\begin{proof}
We can define $\phi$ on $U$ as shown in Fig~\ref{fig:tcm-frames}.
Note that the section of $T^CM$ shown 
as a vector field normal to $C$ and directed towards $C$
can be extended to a continuous section of $T^CM$.
This section maps to the vector field in $N$ normal to $f(C)$.
For vectors in $T^CM$ which are parallel to $C$ we define $\phi$ as $df$.
%
\end{proof}

\begin{proposition}\label{pr:w-tcm}
For a codimension 1 closed submanifold $C\subset M$ we have $w_1(T^C M)=w_1(M)+[C]$.
\end{proposition}

\begin{proof}
The lemma is a restatement of the following fact:
Take a closed curve $\alpha\subset M$ transversal to $C$.
Then $\alpha$ preserves an orientation of $T^C M$
if and only if
either $\alpha$ preserves an orientation of $TM$ and $|\alpha\cap C|\equiv0\mod2$,
or $\alpha$ reverses an orientation of $TM$ and $|\alpha\cap C|\equiv1\mod2$.
This fact follows directly from Definition~\ref{def:tcm}.
\end{proof}

\begin{proposition}\label{pr:ee-tcm}
If $[C]=0$, then $\ee(T^CM)=\pm\big(\chi (M_+)-\chi (M_-)\big)$
where $M_+,M_-\subset M$ are submanifolds with boundary such that
$M=M_+\cup M_-$ and $M_+\cap M_-=\partial M_+=\partial M_-=C$.

If $[C]\ne0$, then $\ee(T^CM)=w_n(M)\equiv \chi(M)-\chi(C)\mod2$.
\end{proposition}

Recall that if $[C]=0$, then the sign of $\ee(T^CM)$ is not defined
as it depends on the choice of the isomorphism $H^n(M;\ZZ_{T^CM})\simeq\Z$,
and if $[C]\ne0$, then $\ee(T^CM)=w_n(M)$ by Proposition~\ref{pr:ee-equals-w-mod2}.

\begin{proof}
The proof is similar to the proof of Proposition~\ref{pr:ee-tm-equals-chi-m}.

Suppose $[C]=0$.
Take triangulations of $M_+$ and $M_-$.
Consider the vector fields on $\partial M_+$ and $\partial M_-$
that consist of unit normal vectors directed into $M_+$, respectively $M_-$.
The total indices of any of their extensions to $M_+$ and $M_-$
are equal to $\chi(M_+)-\chi(C)$, respectively $\chi(M_-)-\chi(C)$.
In order to see this one can take
extensions with a single zero of index $(-1)^k$ at the center of every $k$-simplex that does not belong to $\partial M_+$, respectively $\partial M_-$.
Any two such vector fields together form a section $s$ of $T^C M$.

We have natural isomorphisms $T^CM|_{M_+}\simeq TM|_{M_+}$ and $T^CM|_{M_-}\simeq TM|_{M_-}$.
Therefore a local orientation of $M$ determines orientations of the fibers of $T^CM|_{M_+}$ and $T^CM|_{M_-}$.
Note that when we cross $C$ once, the orientation of a fiber of $T^C M$ changes with respect to the local orientation of $M$.
Therefore the local system $\ZZ_M\otimes\ZZ_E$ changes sign.
So the obstruction $\ee(T^CM)$ is equal to the difference of the total indices of $s|_{M_+}$ and $s|_{M_-}$.
This proves the proposition in the case $[C]=0$.

If $[C]\ne0$, the proof is similar.
Cut $M$ along $C$ and denote the resulting manifold with boundary by $M'$.
If we extend the inward unit normal vector field on $\partial M'$ to the whole $M'$,
then the projection of this extension to $T^CM$ will be a continuous section.
Consider an extension that has a single zero of index $(-1)^k$ at the center of
every $k$-simplex that does not belong to $\partial M'$.
The  total index of this section 
is equal to $\chi(M')-\chi(\partial M')=\chi(M)+\chi(C)-2\chi(C)$.
%
\end{proof}


\subsection{Modifications near cusp points}

In the rest of the paper we take $M,N$ to be closed surfaces.
We also fix a closed submanifold $C\subset M$
and let $P\subset C$ be a discrete subset with given directions. 

In this section we will describe how to construct the vector bundle $f^*(TN)$ where $f:M\to N$ is a $(C,P)$-immersion in terms of $C$ and $P$.

\begin{definition}\label{def:ttb}
The {\it twisted tangent bundle $T^{C,P} M$} is a rank 2 vector bundle over $M$
obtained from $T^C M$ using one of the following two modifications:


1) We remove the tangent bundle $TD_i$ of a small disk $D_i$ near each point $p_i\in P$
as shown in Fig.~\ref{fig:tcpm-definition},~left,
and glue it back using the loop
$t\mapsto\left(\begin{array}{cc} \cos t & -\sin t \\ \sin t & \cos t \\\end{array}\right)$
in $SO(2)$.

2) We remove the tangent bundle $TD_i$ of a small disk $D_i$ near each point $p_i\in P$
as shown in Fig.~\ref{fig:tcpm-definition},~right,
and glue it back using the loop
$t\mapsto\left(\begin{array}{cc} -\cos t & -\sin t \\ \sin t & -\cos t \\\end{array}\right)$
in $SO(2)$.

In both cases the value $t=0$ corresponds to the bottom point of the disk, and as $t$ goes from 0 to $2\pi$ we go around the disk counterclockwise.
\end{definition}

\begin{figure}[h]
\center{\includegraphics
{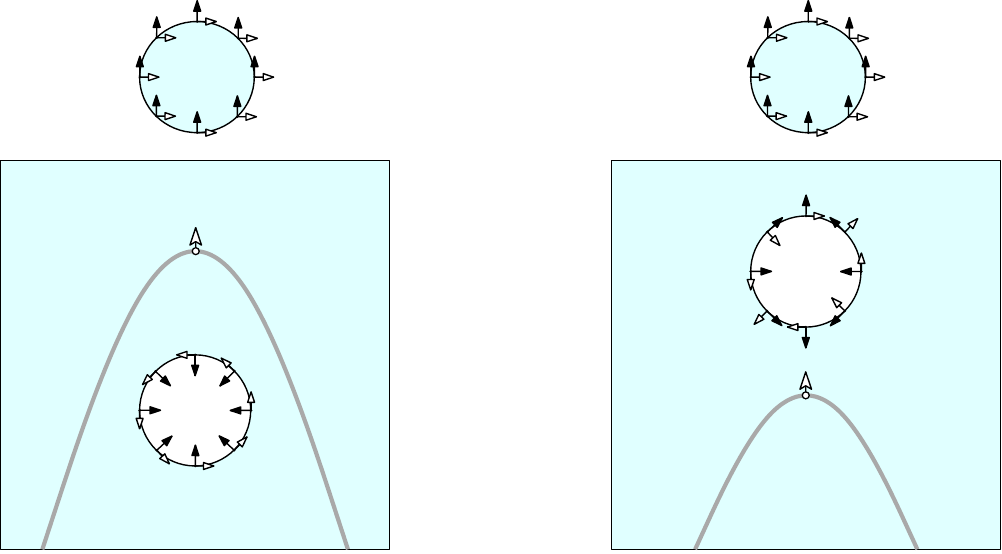}}
\caption{Two types of modification we use to obtain $T^{C,P}M$ from $T^CM$.}\label{fig:tcpm-definition}
\end{figure}

Clearly, in the both cases it does not matter how small and how close to $C$ the disk $D_i$ will be,
as this does not change the bundle up to isomorphism.
Moreover, it turns out that it does not matter whether we use  operation 1 or 2, as both give us isomorphic bundles, see the proof of Proposition~\ref{pr:tcpm-isomorphism}.

As above, 
we identify $TM$ with $T^{C,P}M$ outside $C$.
Also, we will think of $TC$ outside any sufficiently small neighborhood of $P$ as a subbundle of $T^{C,P}M$.

\begin{proposition}\label{pr:tcpm-isomorphism}
If $f:M\to N$ is a $(C,P)$-immersion, then for any tubular neighborhoods $U\supset C$ and $V\supset P$
there exists a fiberwise isomorphism $\phi:T^{C,P} M\to TN$ that covers $f$ and such that
$\phi|_{T(M\setminus U)}=df|_{T(M\setminus U)}$ and $\phi|_{T(C\setminus V)}=df|_{T(C\setminus V)}$
\end{proposition}

\begin{figure}[h]
\center{\includegraphics
{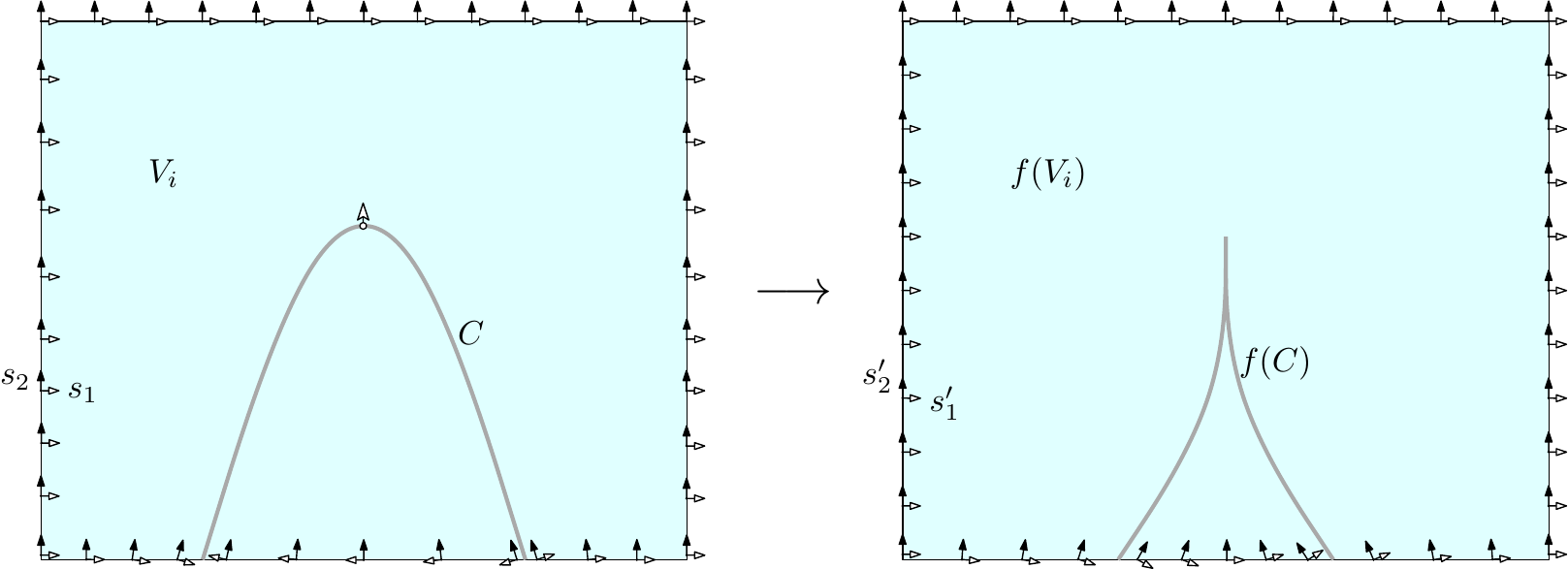}}
\caption{The trivialisation of $TN$ over $\partial f(V_i)$ and its preimage in $T^{C,P}M$.}\label{fig:dv-trivialisation}
\end{figure}

\begin{figure}[h]
\center{\includegraphics
{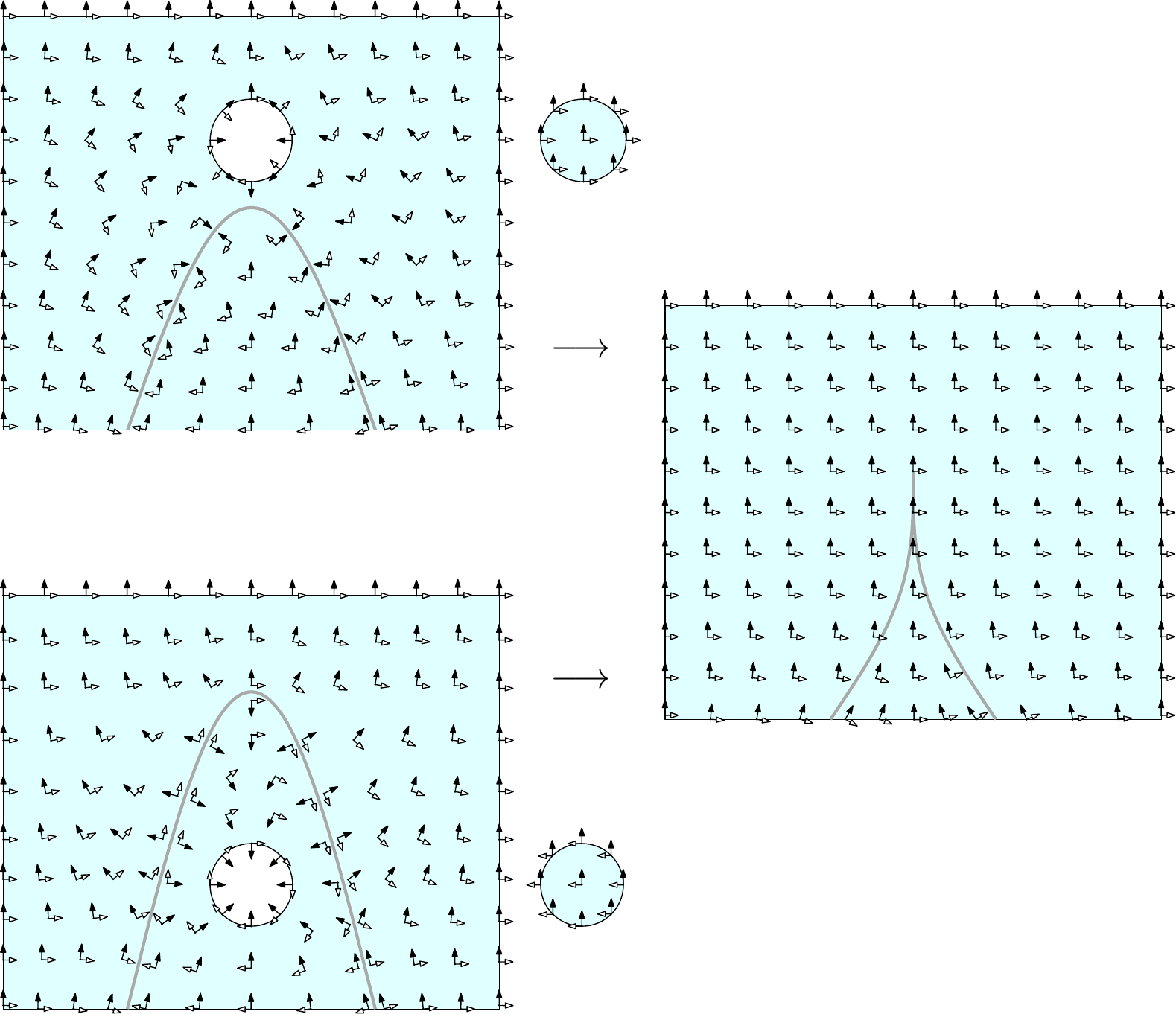}}
\caption{A fiberwise isomorphism $T^{C,P}M|_{V_i}\to TN|_{f(V_i)}$ defined on the trivializations.}\label{fig:v-trivialisation}
\end{figure}

\begin{proof}
Outside $V$ a fiberwise isomorphism $T^{C,P}M|_{M\setminus V}\to TN$ covering $f_{M\setminus V}$
can be constructed as in the proof of Proposition~\ref{pr:tcm-isomorphism}.
We denote it by $\phi^\circ$.

Let $V_i$ be a component of $V$ that contains a cusp point $p_i\in P$.
Consider the field of tangent frames on $ f(\partial V_i)$ shown in Fig.~\ref{fig:dv-trivialisation},~right;
near $f(C)\cap f(\partial V)$ the first vector field $s'_1$ is normal to $f(C)$, and the second one, $s'_2$, is parallel to $f(C)$.
We take their preimages $s_1,s_2$ by the fiberwise isomorphism $\phi^\circ$, see Fig.~\ref{fig:dv-trivialisation},~left.

As shown in Fig.~\ref{fig:v-trivialisation},
we can extend $s_1,s_2$ to the whole of $V_i$ to obtain a trivialization of $T^{C,P}M|_{V_i}$.
The left bottom part of Figure \ref{fig:v-trivialisation} corresponds to the first,
and the left top part to the second, operation used in the definition of $T^{C,P}M$.

Clearly, $s'_1,s'_2$ can be 
extended to the whole of $f(V_i)$
as a trivialization of $TN|_{f(V_i)}$, see Fig.~\ref{fig:v-trivialisation},~right.
Then we define a fiberwise isomorphism $T^{C,P}M|_{V_i}\to TN|_{f(V_i)}$ that covers $f$ by setting $s_1\mapsto s'_1$ and $s_2\mapsto s'_2$.
Note that the resulting fiberwise isomorphism is compatible with $\phi^\circ$ over $\partial V_i$.

To complete the construction of $\phi$ we define it in a similar way over the other components of $V$ and set $\phi|_{M\setminus V}=\phi^\circ$.
\end{proof}

\begin{definition}
A {\it twisted differential} of a $(C,P)$-immersion $f:M\to N$
is a fiberwise isomorphism $T^{C,P}M\to TN$ covering $f$
that coincides with the differential of $f$ outside a sufficiently small neighborhood of $C$,
and whose restriction to $TC$ coincides with the differential of $f|_C$ outside a sufficiently small neighborhood of $P$.
\end{definition}

We have just seen an example of a twisted differential in the proof of Proposition~\ref{pr:tcpm-isomorphism}.

\subsection{Characteristic classes of the twisted tangent bundle}

Recall that $M$ is a closed surface, 
$C\subset M$ is a closed 1-submanifold and $P\subset C$ is a discrete subset with given directions.

\begin{lemma}\label{l:w-tcpm}
We have $w_1(T^{C,P} M)=w_1(M)+[C]$.
\end{lemma}

\begin{proof}
The lemma follows from Proposition~\ref{pr:w-tcm}.
Indeed, the first Stiefel-Whitney class of a vector bundle over a $CW$-complex is determined
by the restriction of the bundle to the $1$-skeleton of the base.
If we choose a $CW$-structure of $M$ so that $P\cap\sk^1(M)=\es$, then clearly
the restrictions of $T^CM$ and $T^{C,P}M$ to $\sk^1(M)$ are isomorphic. 
Therefore $w_1(T^{C,P}M)=w_1(T^CM)=w_1(M)+[C]$.
\end{proof}

\begin{lemma}\label{l:ee-tcpm}
\vbox{If $[C]=0$, then $\ee(T^{C,P}M)=\pm\big(\chi(M_+)-\chi(M_-)-n_+ +n_-\big)$.
Here $M_+,M_-$ are defined as in Proposition~\ref{pr:ee-tcm} and
$n_+$ (respectively $n_-$) is the number of points of $P$ at which the direction vector is directed into $M_-$ (respectively $M_+$).
Call them positive (respectively negative).

If  $[C]\ne0$, then $\ee(T^{C,P}M)=w_2(T^{C,P}M)\equiv \chi(M)+|P|\mod2$.}
\end{lemma}

\begin{proof}
We will modify the construction from the proof of Proposition~\ref{pr:ee-tcm} to obtain a section of~$T^{C,P}M$.

Namely, suppose $[C]=0$.
Consider the inward unit normal vector field on $\partial M_+$.
Take its extension to $M_+$ with a zero of index $1$ near every positive element of $P$
(and possibly with some other zero points).
Similarly, extend the inward unit normal vector field from $\partial M_-$ to the whole of $M_-$
so that the extension has a single zero point of index $1$ near every negative element of $P$.
Let $P'$ be the set of the zeroes of the extended vector fields near points of $P$.
Note that $|P'|=|P|=n_++n_-$.

The total indices of these vector fields 
are equal to $\chi(M_+)-\chi(C)$, respectively $\chi(M_-)-\chi(C)$.
Together these vector fields form a section $s$ of $T^CM$.

Recall that the elements of $P'$ are ``inside'' the cusp with respect to given direction.
Figure~\ref{fig:tcpm-definition},~left, shows that we can modify $s$ in a neighborhood of $P'$
so as to get a section $\hat s$ of $T^{C,P}M$.
Note that $\hat s$ has $n_+$ less zeroes (respectively $n_-$ less zeroes) of index $1$
inside $M_+$ (respectively inside $M_-$).
So we can see that the total index of $\hat s|_{M_+}$ equals $\chi(M_+)-\chi(C)-n_+$,
and the total index of $\hat s|_{M_-}$ equals $\chi(M_-)-\chi(C)-n_-$.
Therefore the first part of the lemma follows.

To prove the lemma in the case $[C]\ne0$ one can modify the corresponding construction
from the proof of Proposition~\ref{pr:ee-tcm} in a similar way to the above.
\end{proof}

\section{Proof of Theorem \ref{th:homotopic}}
\label{s:proof-homotopic}

\begin{proof}

The necessity of all conditions of Theorem \ref{th:homotopic} follows
from Proposition~\ref{pr:tcpm-isomorphism} and Lemma~\ref{l:ee-tcpm} via functoriality of Stiefel-Whitney and Euler classes.
It remains to show that these conditions are sufficient.

\vspace{.5em}

Suppose we are given $M$, $N$, $C$, $P$ and $f$ for which conditions 1.1--1.3 from Theorem~\ref{th:homotopic}
hold.
By Lemmas~\ref{l:ee-universal} and \ref{l:ee-tcpm} the vector bundles $T^{C,P}M$ and $f^*(TN)$ are isomorphic.
Let $\psi: T^{C,P}M \to TN$ be a fiberwise isomorphism covering $f$.

Take a small compact neighborhood $V$ of $P$ that deformation retracts onto $P$.
We can deform $f$ in $V$ so that
it becomes smooth on $V$ and has
the prescribed singularities (namely, folds in $(C\cap V)\setminus P$ and cusps with given directions in $P$),
and so that for every $x\in V\setminus C$
the isomorphisms of the fibers $df|_{x},\psi|_{x}\in\Iso(T_xM,T_{f(x)}N)$ belong to the same component.
%
Then
we can also deform $\psi$
in the class of fiberwise isomorphisms covering $f$
so that the resulting $\psi|_V$ will be a twisted differential of $f|_V$.

Using the Hirsch's Theorem~\ref{th:hirsch} we can deform $f$ and $\psi$ stationary on $V$
so that $f|_{C\setminus P}$ becomes an immersion
and $\psi|_{C\setminus V}$ becomes equal to $d(f|_{C\setminus V})$.
Fig.~\ref{fig:df-on-c} shows how to do this when the images $\psi(TC)$ and $df(TC)$ differ on some segment of $C$.

\begin{figure}[h]
\center{\includegraphics
{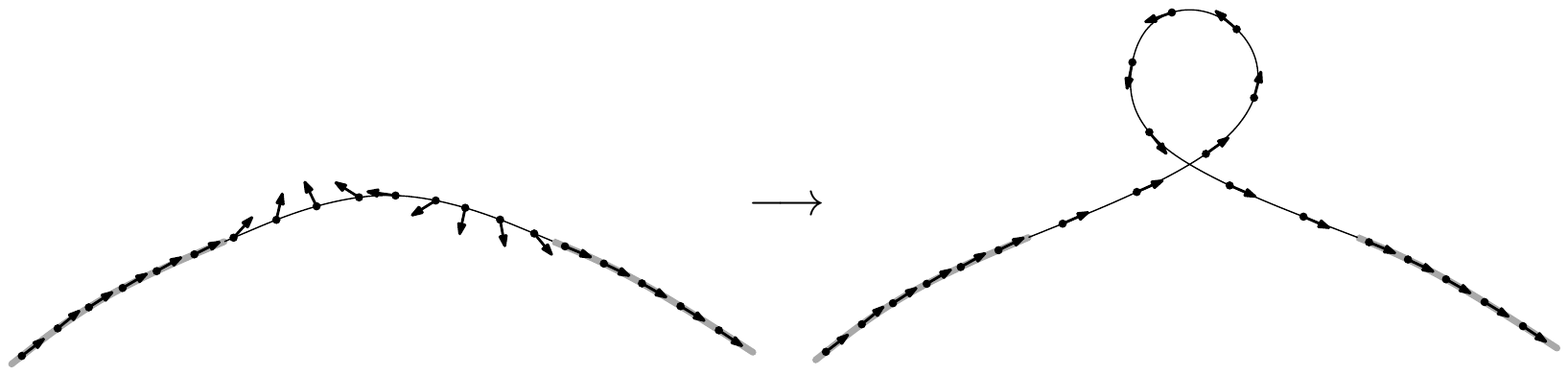}}
\caption{Deformation of $f$ to a map whose restriction to $C$ has a prescribed differential.}\label{fig:df-on-c}
\end{figure}

Take a tubular neighborhood $U\supset C$.
We can deform $f$ stationary on $V\cup C$ 
so that it becomes smooth in $U$ and $f|_U$ turn to a $(C,P)$-immersion.
Then we deform $\psi$ stationary on $V\cup C$ so that $\psi|_U$ becomes a twisted differential of $f|_U$.
We can do this since $U$ deformation retracts onto $C$.

As a result, $\psi|_{\partial U}=df|_{\partial U}$.
Define the morphism $\phi:TM\to TN$ as $df$ in $U$ and as $\psi$ outside~$U$.
Then $\phi|_{T(M\setminus V)}$ is a $C$-monomorphism.
We can now complete the proof using Eliashberg's Theorem~\ref{th:eliash}.
\end{proof}

\section{Existence of maps with prescribed folds and cusps}

\subsection{Surfaces and curves}\label{s:curves}

The following statements about surfaces and curves will be used in the proof of Theorem~\ref{th:existence}.

\begin{proposition}\label{pr:basis-curves}
For a closed nonorientable surface $S$ of nonorientable genus $g$ there is a basis $b_1,\ldots,b_g\in H^1(S,\Z_2)$
such that $b_i^2=w_2(S)$ and $b_i\smile b_j=0$ for $i\ne j$.
\end{proposition}

\begin{proof}
Note that $S$ can be viewed as the connected sum of $g$ copies of $\R P^2$.
Let $C_i$ be a curve in the $i$-th $\R P^2$ that changes the orientation.
Then we set $b_i=[C_i]$.
\end{proof}

Recall that a closed curve is called {\it simple} if it has no self-intersections.

\begin{proposition}\label{pr:independent-curves}
Let $S$ be a closed surface.
Suppose $a_1,\ldots,a_k\in H^1(S,\Z_2)$ are linearly independent elements such that $a_i\smile a_j=0$ for $i\ne j$.
Then there are disjoint simple closed curves $C_i\subset S,\ i=1,\ldots,k$ such that $[C_i]=a_i$.
\end{proposition}

\begin{figure}[h]
\center{\includegraphics{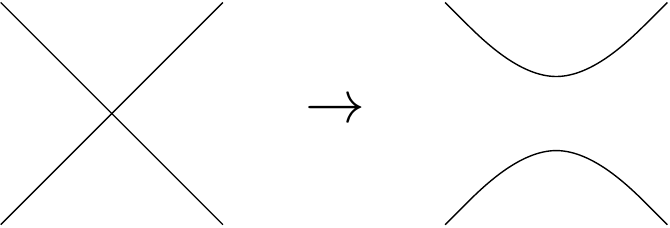}}
\caption{Resolution of a crossing.}\label{fig:resolution-of-crossing}
\end{figure}

\begin{proof}
Every homology class of $S$ can be represented by a union of closed curves.
One can remove all self-intersection points using the resolution shown in Fig.~\ref{fig:resolution-of-crossing}.
Then, since $S$ is connected, using the surgery shown in Fig.~\ref{fig:resolution-connecting}
we can transform the resulting union into a single curve.
Note that these operations do not change the homology class.
This shows that we can choose a collection of simple closed curves $C_i\subset S,\ i=1,\ldots,k$ in general position such that $[C_i]=a_i$.

\begin{figure}[h]
\center{\includegraphics{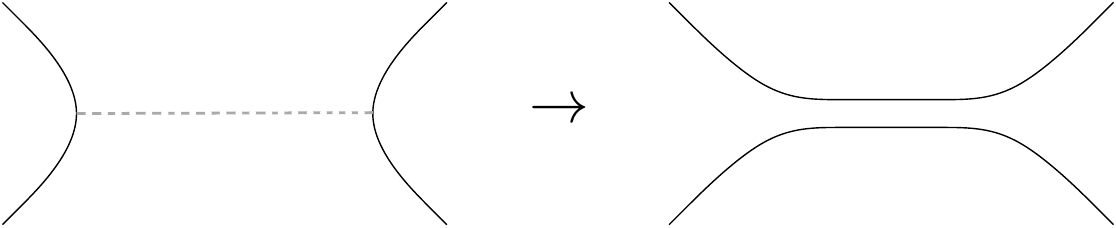}}
\caption{Surgery of a curve.}\label{fig:resolution-connecting}
\end{figure}

For all $i\ne j$ the curve $C_i$ intersects $C_j$ at an even number of points, since $a_i\smile a_j=0$.
We can remove all intersection points by induction.
Suppose $C_1,\ldots,C_{i-1}$ are disjoint simple closed curves.
Using the resolution shown in Fig.~\ref{fig:resolution-of-two-crossings}
we can change $C_i$ to a collection $C'$ of simple closed curves disjoint with all $C_j,\ j<i$.
Since the elements $a_1,\ldots,a_{i-1}$ are linearly independent,
the union $C_1\cup\ldots\cup C_{i-1}$ does not separate $S$.
Therefore via the operation from Fig.~\ref{fig:resolution-connecting} we can transform $C'$ to a single simple closed curve.
\end{proof}

\begin{figure}[h]
\center{\includegraphics{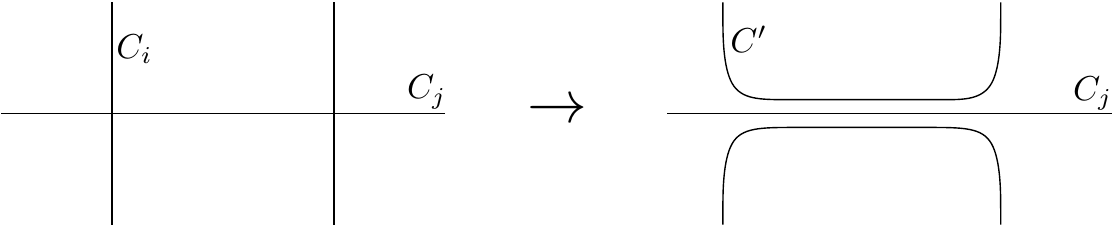}}
\caption{Resolution of two neighboring crossings.}\label{fig:resolution-of-two-crossings}
\end{figure}

Recall that {\it a crosscap} in a surface $S$ is an embedded M\"{o}bius band.
By a {\it crosscap decomposition} we mean a representation of a nonorientable surface as the sphere with attached crosscaps
(or equivalently, as the connected sum of several copies of $\R P^2$).
If some curve $C\subset S$ transversally intersects the middle line 
of an embedded M\"{o}bius band at a single point,
then we say that {\it $C$ goes through the crosscap}.

\begin{proposition}\label{pr:curve-through-crosscaps}
Let $S$ be a closed nonorientable surface of nonorientable genus $g$.
Let $C\subset S$ be a simple closed curve that changes the orientation.
Then either for any crosscap decomposition of $S$ the curve $C$ goes through every crosscap,
or for any odd $k<g$ there is a crosscap decomposition of $S$ such that $C$ goes through $k$ crosscaps.
\end{proposition}

\begin{proof}
If $[C]=w_1(M)$, then the 2-manifold $S\setminus C$ is orientable.
Therefore there is no crosscap in $S$ disjoint with $C$.

If $[C]\ne w_1(M)$, then $S\setminus C$ is a nonorientable surface of nonorientable genus $g-1$ with one puncture.
Take another closed nonorientable surface $S'$ of nonorientable genus $g$ and
a simple closed curve $C'\subset S'$ that goes through $k$ crosscaps, where $k$ is odd and $k<g$.
Note that the surfaces $S\setminus C$ and $S'\setminus C'$ are homeomorphic.
So there is a homeomorphism $S'\to S$ that maps $C'$ to~$C$.
It remains to take the crosscap decomposition of $S$ induced from $S'$.
\end{proof}

\begin{figure}[h]
\center{\includegraphics
{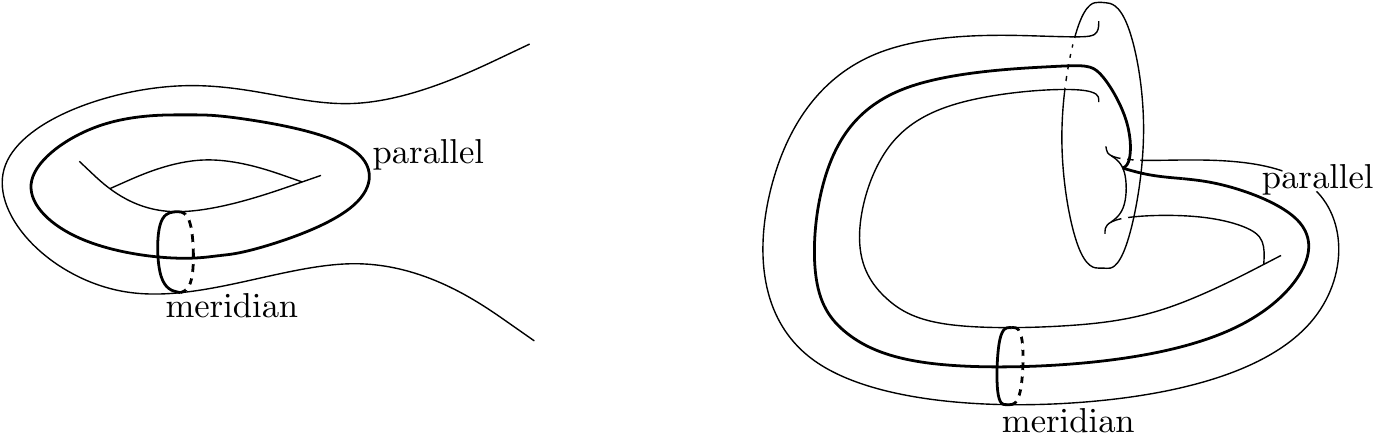}}
\caption{A parallel and a meridian of an orientable and of a nonorientable handle.}\label{fig:parallel-meridian}
\end{figure}

\begin{proposition}\label{pr:orientation-of-curves-maps}
Suppose $M,N$ are closed surfaces and $C\subset M$ is a meridian of a handle
(possibly orientable or not, see Fig.~\ref{fig:parallel-meridian}).
Let $f$ be the composition of a retraction of $M$ to a parallel of this handle and
a map of the parallel to some curve in $N$ that changes the orientation.
Then $f^*w_1(N)=[C]$.
\end{proposition}

\begin{proof}
Here our strategy is very similar to the proof of Proposition~\ref{pr:w-tcm}.
Namely, take a curve $C'\subset M$ in general position with $C$.
Then $f(C')$ changes the orintation in $N$ if and only if $|C\cap C'|$ is odd,
and this holds if and only if $[C]\frown C'=1$.
\end{proof}

\begin{lemma}\label{l:genus-realizing}
Let $M,N$ be closed surfaces and $d\in\Z_{\ge0}$.
There exists a map $f:M\to N$ such that $f^*w_1(N)=w_1(M)$ and $\deg f=d$ if and only if
$\chi(M)\le d\cdot\chi(N)$
and one of the following conditions holds:
\begin{itemize}
\item both $M$ and $N$ are orientable;
\item $M$ is orientable, $N$ is nonorientable and $d$ is even;
\item both $M$ and $N$ are nonorientable and $\chi(M)\equiv d\cdot\chi(N)\mod2$.
\end{itemize}
\end{lemma}

\begin{proof}
For closed orientable surfaces the bound $\chi(M)\le d\cdot\chi(N)$
is a well-known fact, see \cite{kneser} and \cite{edmonds}.
One can generalize this bound to possibly nonorientable $M,N$
by considering the orientable 
covers $\widetilde M\to M$ and $\widetilde N\to N$
and the corresponding map  $\widetilde f:\widetilde M\to\widetilde N$.

\vspace{.5em}

Suppose both $M$ and $N$ are orientable.
Every map of degree $d$ such that $\chi(M)\le d\cdot\chi(N)$
can be realized as the composition of a map that collapses some handles in $M$ and a (possibly branched) covering map.

\vspace{.5em}

Suppose $M$ is orientable and $N$ is nonorientable.
Then every $f:M\to N$ such that $f^*w_1(N)=w_1(M)$ can be lifted to the orientable double cover $\widetilde N\to N$, so $\deg f$ must be even.
To find a degree $d$ map $M\to N$
we take the composition of a degree $\frac d2$ map $M\to\widetilde N$ from the previous case
and the projection $\widetilde N\to N$.

\vspace{.5em}

Suppose $M$ is nonorientable and $N$ is orientable.
Then there are no maps $f:M\to N$ such that $f^*w_1(N)=w_1(M)$.

\vspace{.5em}

Finally, suppose both $M$ and $N$ are nonorientable.
Note that for every closed surface $S$ we have $w_1(S)^2=w_2(S)$.
Indeed, $w_1(S)^2\ne0$ if and only if $S$ is a connected sum of an odd number of copies of $\R P^2$, which is equivalent to $w_2(S)\ne 0$.
(The equality $w_1(M)^2=w_2(M)$ can also be deduced from Wu's formula, see e.\,g.~\cite[p.\,131]{milnor-stasheff}.)
So if $f^*w_1(N)=w_1(M)$, then $w_2(M)=f^*w_2(N)$.
Therefore if $w_2(N)=0$, then $w_2(M)=0$.
And if $w_2(N)\ne0$, then by Proposition~\ref{pr:ee-equals-w-mod2} we have
$w_2(M)=\deg f\cdot \alpha$, where $\alpha\in H^2(M;\Z_2)$ is the generator.
We have proved that $\chi(M)\equiv d\cdot\chi(N)\mod2$.

It remains to prove the `if' part in the case when $M$ and $N$ are both nonorientable.
We consider the following three subcases.

Suppose $N=\R P^2$ and $d$ is a positive integer such that $d\equiv\chi(M)\mod2$.
Let $M'$ be a closed orientable surface such that $M=M'\#\R P^2$ or $M'\#\R P^2\#\R P^2$,
depending on whether $\chi(M)$ is odd or even.
%
If we collapse the circles along which the components of the connected sum are glued,
then we obtain a projection to the wedge sum.
So, let $f_1$ be the projection of $M$ to $M'\vee \R P^2$ (respectively to $M'\vee \R P^2\vee\R P^2$).
We set $f_2:M'\vee \R P^2\to N$ (respectively $f_2:M'\vee \R P^2\vee\R P^2\to N$)
to be a homeomorphism on $\R P^2$'s and an arbitrary map of degree ${d-1}$ (respectively ${d-2}$) on $M'$.
Such a map exists since $d-1$ (respectively $d-2$) is even.
Then $f=f_2\circ f_1$ has degree $d$ and $f^*w_1(N)=w_1(M)$.

Suppose $\chi(N)\le0$ and $d$ is a positive integer such that $\chi(M)\le d\cdot\chi(N)$ and $\chi(M)\equiv d\cdot\chi(N)\mod2$.
Then $M$ can be viewed as the connected sum of
a closed non\-orientable surface $M'$ and a closed orientable surface $M''$
such that $\chi(M')=d\cdot\chi(N)$ and $\chi(M'')=\chi(M)-d\cdot\chi(N)+2$.
We set $f$ to be the composition of the projection $M\to M'$ and a $d$-fold covering $M'\to N$.
Then $f$ has degree $d$ and $f^*w_1(N)=w_1(M)$.

Finally, suppose 
$d=0$.
Since $\chi(M)$ is even, $M$ can be viewed as the connected sum of the Klein bottle $K$ and a closed orientable surface $M'$.
Let $C\subset K$ be any parallel, 
i.\,e.\ a simple closed curve that changes the orientation.
We set $f$ to be the composition of a retraction $M\to C$ and a map of $C$ to some curve in $N$ that changes the orientstion.
Then $f^*w_1(N)=w_1(M)$ 
by Proposition~\ref{pr:orientation-of-curves-maps}.
\end{proof}

\subsection{Proof of Theorem \ref{th:existence}}\label{s:proof-existence}

\begin{proof}
First we prove the `only if' part.
Suppose there is a $(C,P)$-immersion $f:M\to N$.
We need to check that $M,N,C,P$ satisfy conditions 2.1--2.4 of Theorem~\ref{th:existence}.

By Theorem~\ref{th:homotopic} we have $[C]=w_1(M)+f^*w_1(N)$ and $[P]=w_2(M)+f^*w_2(N)$.
Recall that for every closed surface $S$ we have $w_1(S)^2=w_2(S)$.
So we have
$$[P] = w_2(M)+f^*w_2(N) = w_1(M)^2 + f^*(w_1(N)^2) = w_1(M)^2 + ([C]+w_1(M))^2 = [C]^2.\eqno (*)$$
This gives us condition 2.1.

Conditions 2.2 and 2.3 follow from conditions 1.1 and 1.3 of Theorem~\ref{th:homotopic}.
We set $d=\deg f$.
Then the inequality $\chi(M)\le |d|\cdot\chi(N)$ and the congruence $\deg f\equiv0\mod2$ in case $M$ is orientable and $N$ is nonorientable
hold in view of Lemma \ref{l:genus-realizing}.

Suppose $M,N,C,P$ satisfy the hypotheses of condition 2.4, i.\,e.\ $[C]\ne0$, $w_1(M)\ne0$ and $[C]^2\ne w_2(M)$.
It follows from conditions 1.2 and 2.1 that $f^*w_2(N)\ne0$.
So $w_2(N)\ne0$.

By Proposition~\ref{pr:basis-curves}, there exists a basis $b_1,\ldots,b_g\in H^1(N,\Z_2)$
such that $b_i^2=w_2(N)$ and $b_i\smile b_j=0$ for $i\ne j$.
Denote the elements $f^*(b_i)\in H^1(M,\Z_2)$ by $a_i$.
Since $f^*w_2(N)\ne0$, we have $a_i^2\ne0$ and $a_i\smile a_j=0$ for $i\ne j$.
Therefore none of $a_i$ does not equal to the linear combination of the others.

Then by Proposition~\ref{pr:independent-curves}
one can choose a collection of disjoint simple closed curves $C_i\subset M,\ i=1,\ldots,g,$ such that $[C_i]=a_i$.
A tubular neighborhood of every $C_i$ is a crosscap.
By condition~1.1 we have
$$\sum_{i=1}^g a_i = \sum_{i=1}^g f^*(b_i) = f^*w_1(N) = w_1(M)+[C],$$
but $w_1(M)+[C] \ne w_1(M)$ by the hypotheses of condition~2.4.
Therefore 
$\sum_{i=1}^g a_i \ne w_1(M)$, so $M$ has more then $g$ crosscaps,
which implies $\chi(N)>\chi(M)$.
This completes the proof of condition 2.4.

\vspace{.5em}


The proof of the `if' part also uses Theorem~\ref{th:homotopic}.
Namely, for $M,N,C,P$ which satisfy conditions 2.1--2.4 of Theorem~\ref{th:existence}
we will construct a map $f:M\to N$ that satisfies the conditions of Theorem~\ref{th:homotopic}.
In each of the following cases conditions 1.1 and 1.3 are easy to check,
and condition 1.2 follows from 1.1 and 2.1, see $(*)$.

I. {\it $N$ is orientable}.
By condition 2.2 we have $[C]=w_1(M)$, so condition 1.1 is allways satisfied.
If $M$ is nonorientable, we need not to check condition 1.3, we can take $f$ to be a null-homotopic map.
If $M$ is orientable, set $f$ to be any map of degree $d$.
Such a map exists by Lemma \ref{l:genus-realizing}.
Then condition 1.3 follows from 2.3.

II. {\it $N$ is nonorientable, $w_1(M)=[C]=0$}.
As in the previous case, set $f$ to be any map of degree~$d$.
We can apply Lemma \ref{l:genus-realizing} since $d$ is even by condition 2.3.

III. {\it $N$ is nonorientable, $w_1(M)=[C]\ne0$}.
We can take $f$ to be a null-homotopic map; the details are  similar to the first case.

IV. {\it $N$ is nonorientable, $[C]\ne0=w_1(M)$}.
In this case there exists a homeomorphism $f_1$ from $M$ to the standard orientable surface $M'\subset\R^3$
such that the image of $C$ is homologous to a meridian $C'$ of a handle (see Fig.~\ref{fig:parallel-meridian},~left).
Let $f_2$ be a retraction of $M'$ to a parallel of this handle and
let $f_3$ be a map of the parallel to some curve in $N$ that changes the orientation.
Set $f$ to be the composition $f_3\circ f_2\circ f_1$.
In this case we do not need to check condition 1.3,
and condition 1.1 
follows from Proposition~\ref{pr:orientation-of-curves-maps}.


V. {\it $N$ is nonorientable, $[C]=0\ne w_1(M)$}.
Note that $|P|=n_+ + n_- \equiv n_+ - n_-\mod2$,
so $n_+ - n_- \equiv 0\mod2$ by condition 2.1.
Using condition 2.3 we can see that $\chi(M)-d\cdot\chi(N)\equiv0\mod2$.
Then by Lemma \ref{l:genus-realizing} there is a degree $d$ map $f:M\to N$.
For this map conditions 1.1 and 1.3 hold by construction.

VI. {\it $N$ is nonorientable, $0\ne[C]\ne w_1(M)\ne0$}.
Let $C'$ be a simple closed curve in $M$ such that $[C']=[C]+w_1(M)$.
The curve $C'\subset M$ either preserves the orientation, or changes the orientation.

If $C'$ preserves the orientation, then $M\setminus C'$ is a surface with two punctures.
Therefore $C'\subset M$ is a meridian of some orientable or nonorientable handle, see Fig.~\ref{fig:parallel-meridian}.
This means that $M$ can be viewed as a connected sum of the torus, respectively Klein bottle, and some other closed surface.
We set $f$ to be the composition of a retraction of $M$ to a parallel of the handle and
a map of the parallel to some curve in $N$ that changes the orientation.
Then condition 1.1 
follows from Proposition~\ref{pr:orientation-of-curves-maps}.

If $C'$ changes the orientation, then $0\ne[C']^2=[C]^2+w_1(M)^2=[C]^2+w_2(M)$.
Therefore our data satisfies the hypotheses of condition 2.4,
so the nonorientable genus of $N$, call it $k$, is odd and it is strictly less than the nonorientable genus of $M$.
By Proposition \ref{pr:curve-through-crosscaps} there is a crosscap decomposition of $M$ such that $C'$ goes through $k$ of the crosscaps.
If we collapse the other crosscaps in $M$, then the resulting surface $M'$ is homeomorphic to $N$ and the pullback of $w_1(M')$ equals $[C']$.
Then we set $f$ to be the composition of the projection $M\to M'$ and a homeomorphism $M'\to N$.
\end{proof}


Faculty of Mathematics
\nopagebreak

National Research University
\nopagebreak

Higher School of Economics
\nopagebreak

Moscow, Russia
\nopagebreak

\verb"ryabichev@179.ru"


\begin{thebibliography}{99}




\bibitem{edmonds}
A.\,Edmonds,
Deformation of maps to branched coverings in dimension two.
Ann. Math., vol. 110 (1979), 113--125.

\bibitem{eliashberg}
Y.\,M.\,Eliashberg,
On singularities of folding type.
Math. USSR Izv., 4 (1970), 1119--1134.

\bibitem{hmjrf}
D.\,Hacon, C.\,Mendes de Jesus, M.\,C.\,Romero Fuster,
Graphs of stable maps from closed orientable surfaces to the 2-sphere.
J. Singul. 2 (2010), 67--80.

\bibitem{hirsch}
M.\,Hirsch,
Immersions of manifolds.
Trans. Amer. Math. Soc., vol. 93 (1959), 242--276.

\bibitem{kneser}
H.\,Kneser,
Die kleinste Bedeckungszahl innerhalb einer Klasse von Fl\"{a}chenabbildungen.
Math. Ann. 103 (1930), 347--358.

\bibitem{mj}
C.\,Mendes de Jesus,
Graphs of stable maps between closed orientable surfaces.
Comput. Appl. Math. 36 (2017), no. 3, 1185--1194

\bibitem{mjrf}
C.\,Mendes de Jesus, M.\,C.\,Romero-Fuster,
Graphs of stable maps from closed surfaces to the projective plane.
Topology Appl. 234 (2018), 298--310.

\bibitem{milnor-stasheff}
J.\,Milnor, J.\,D.\,Stasheff,
Characteristic Classes.
Princeton University Press, 1974.

\bibitem{prasolov}
V.\,V.\,Prasolov,
Elements of Combinatorial and Differential Topology.
Graduate Studies in Mathematics 74, 2006.

\bibitem{spanier}
E.\,Spanier,
Singular homology and cohomology with local coefficients and duality for manifolds.
Pacific Journal of Mathematics, vol. 160, no. 1, (1993), 165--200.

\bibitem{whitehead}
G.\,Whitehead,
Elements of homotopy theory.
Springer-Verlag, 1978.

\bibitem{whitney}
H.\,Whitney,
On singularities of mapping of Euclidean spaces. I. Mappings of the plane into the plane.
Annals of Mathematics, vol. 62, no. 3, (1955), 374–-410.

\bibitem{m.yamamoto}
M.\,Yamamoto,
The number of singular set components of fold maps between oriented surfaces.
Houston J. Math. 35 (2009), no. 4, 1051--1069.

\bibitem{yamamoto}
T.\,Yamamoto,
Apparent contours of stable maps between closed surfaces.
Kodai Math J., vol. 40 (2017), 358--378.

\end{thebibliography}
\end{document}